\documentclass[11pt]{amsart}

\usepackage[usenames,dvipsnames]{pstricks}
\usepackage{times}
\usepackage{amsmath,verbatim,graphicx,epstopdf,enumerate}
\newcommand{\D}{\mathrm{d}}

\newcommand{\lb}{\left(}

\newcommand{\rb}{\right)}
\newcommand{\PD}{\partial}

\newcommand{\Sb}{\mathbb{S}}

\newcommand{\Beq}{\begin{equation}}
	\newcommand{\Eeq}{\end{equation}}
\newcommand{\beq}{\begin{equation*}}
	\newcommand{\eeq}{\end{equation*}}
\newcommand{\bal}{\begin{align}}
	\newcommand{\eal}{\end{align}}

\newcommand{\n}{\nabla}

\newcommand{\bp}{\begin{prob}}
\newcommand{\bpr}{\begin{proof}}
	\newcommand{\epr}{\end{proof}}

\newcommand{\bel}[1]{\begin{equation}\label{#1}}
	\newcommand{\ee}{\end{equation}}

\newtheorem{theorem}{Theorem}[section]
\newtheorem{corollary}[theorem]{Corollary}

\newtheorem{lemma}[theorem]{Lemma}

\theoremstyle{definition}

\newtheorem{remark}[theorem]{Remark}

\newcommand{\R}{{\mathbb R}}

\newcommand{\C}{{\mathbb C}}

\newcommand{\be}{\begin{eqnarray}}
\newcommand{\ben}{\begin{eqnarray*}}
\newcommand{\en}{\end{eqnarray}}
\newcommand{\enn}{\end{eqnarray*}}

\newcommand{\real}{{\rm Re\,}}

\newcommand{\mm}[1]{{\color{black}{#1}}}
\definecolor{rot}{rgb}{0.000,0.000,0.000}
\newcommand{\tcr}{\textcolor{rot}}

\title[ Uniqueness for time-dependent inverse problems with single dynamical data ]{ Uniqueness for time-dependent inverse problems
\\ \mm{with single dynamical data}}
\author[Ben A\"icha, Hu, Vashisth and Zou]{ Ibtissem Ben A\"icha$^{1}$, Guang-Hui Hu$^2$,  Manmohan Vashisth$^{3}$ and Jun Zou$^{4}$}

\address{  $^{1}$ Beijing Computational Science Research Center, Beijing 100193, China.
	\newline
	\indent E-mail:{\tt \ ibtissem@csrc.ac.cn}}
\address{ $^2$ Beijing Computational Science Research Center, Beijing 100193, China.
	\newline
	\indent E-mail:{\tt \ hu@csrc.edu}}
\address{$^3$ Beijing Computational Science Research Center, Beijing 100193, China.
	\newline
	\indent E-mail:{\tt\  mvashisth@csrc.ac.cn}}
	\address{$^4$ Department of Mathematics, The Chinese University of Hong Kong, Hong Kong, Hong Kong Special Administrative Region of China.
	\newline
	\indent E-mail:{\tt\  zou@math.cuhk.edu.hk}}	
\begin{document}
%


	\maketitle
\begin{abstract}
	\mm{In this work, we investigate the shape identification and coefficient determination associated with
	two} time-dependent partial differential equations in two dimensions. We consider the inverse problems of determining \tcr{a convex polygonal obstacle} and the  coefficient  appearing in  the wave and Schr\"odinger equations from a single dynamical data
	\mm{along with the time}.
	\mm{With the far field data, we first prove} that the sound speed of the wave equation \tcr{together with its contrast support of convex-polygon type} can be uniquely determined, \mm{then establish} a uniqueness result
	for recovering an electric potential as well as its support appearing in the Schr\"odinger equation.
	As a consequence of these results,  we \mm{demonstrate} a uniqueness result for recovering the refractive index of a medium from a single far field pattern \tcr{at a fixed frequency in the time-harmonic regime}.

\end{abstract}
	


\maketitle


\section{Introduction}
\noindent \mm{This work aims at the mathematical understanding of the unique identifiability
for three coefficient/obstacle inverse problems concerning the wave, the Schr\"odinger and the Helmholtz equations.
We first introduce} some notations which will be used throughout \mm{the work.}
For  $R>0$, \mm{we shall write} $B_R:=\{x\in \R^{2}: |x|<R\}$, $\Gamma_{R}:=\{x\in\R^{2}:\ \lvert x\rvert =R\}$,
and $\tcr{\Sb}:=\{x\in\R^{2}:\ \lvert x\rvert =1\}$.
\subsection{Formulation of \mm{the inverse wave problem}}\label{Sub:wave equtaion formulation}
Consider the propagation of acoustic waves in an unbounded inhomogeneous background medium due to a compactly supported source term in $\R^2$. This can be modeled by the \tcr{inhomogenous} wave equation
\be\label{wave-equation}
\frac{1}{c^2(x)}\frac{\partial^2 u}{\partial t^2}=\Delta u+ f(x)g(t)\quad\mbox{in}\quad \R^2\times \R_+,
\en together with the initial conditions
\be\label{Initial-wave}
u(x,0)=\partial_t u(x,0)=0\quad\mbox{on}\quad \R^2.
\en
For $c\in \tcr{\mathcal{C}}(\R^{2})$, $f\in L^{2}(\R^{2})$ and $g\in L^{1}(\R_{+})$,
\mm{the initial value problem} \eqref{wave-equation}-\eqref{Initial-wave} admits a unique solution $u$ such that $u\in \tcr{\mathcal{C}}\lb \R_{+};H^{1}(\R^{2}\rb \cap \tcr{\mathcal{C}}^{1}\lb\R_{+};L^{2}(\R^{2})\rb$.
In the first part of the paper, we study the inverse problem of determining the sound speed $c$ \mm{in Equation \eqref{wave-equation}}
and the shape of \mm{an unknown inhomogeneous medium, namely,
supp$\lb 1-c(x)\rb$,}  from the knowledge of a single dynamical data.
More precisely, we prove \mm{that the sound speed and its support are unique, under}
the knowledge of the solution $u$ measured on $\Gamma_{R}\times \R_{+}$, \tcr{provided they satisfy some a priori conditions (see Section \ref{Sec:Statement main result} below).}

\subsection{Formulation of \mm{the inverse Schr\"odinger problem}}\label{Sub:Schroedinger equtaion formulation}
The second  part of the \mm{work} deals with two inverse problems \mm{arising from}
the Schr\"odinger equation defined in $\R^2\times \R_{+}$. More precisely, \mm{we intend} to uniquely determine a compactly supported electric potential as well as its convex polygonal support, from the knowledge of the boundary observation.   \mm{We shall consider} the time-dependent  Schr\"odinger equation
\begin{equation}\label{Eq1.1}
     (i\PD_{t}u+\Delta +q(x))u(x,t)=0 \qquad  \mbox{in}\,\, \R^2\times \R_{+},
     \end{equation}
together with the initial condition
 \begin{equation}\label{initial-Schrodinger}
 u(x,0)=u_{0}\qquad \mbox{in}\,\, \R^2,
  \end{equation}
where  $u_0\in H^{2}(\R^2)$ and $q\in \mathcal{C}(\R^2)$ is the compactly supported electric potential that is assumed to be
\mm{a real-valued function.}  According to \cite{[Lions]}\color{black}, \mm{the initial value problem (\ref{Eq1.1})-(\ref{initial-Schrodinger}) is well posed, with a unique solution $u\in \mathcal{C}(\R_{+}; H^2(\R^2))$,
also satisfying \tcr{the energy identity}
}
  \begin{equation}\label{energy consevation}
  \|u(\cdot,t)\|_{L^2(\R^2)}=\|u_0\|_{L^2(\R^2)}\tcr{\quad\mbox{for all}\quad t>0.}
  \end{equation}
 Our goal in the second part of the paper is to deal with the inverse problem of determining the electric potential $q$ as well as its polygonal support $D\!:=\!\mbox{supp}\,(q)\subset B_R$ from the knowledge of the boundary measurement $u_{|\Gamma_{R}\times \R_+}$.
\subsection{Formulation of \mm{the inverse Helmholtz problem}}
Consider the time-harmonic medium scattering problem \tcr{for the total field \mm{$v=v^{in}+v^{sc}$}}:
\begin{equation}\label{Equation for v time-harmonic}
\Delta v+k^2 n(x)v=0\quad\mbox{in}\quad \R^2\,, 
\end{equation}
\mm{where $k>0$ is the wave number of the homogeneous  isotropic background medium,}
and the refractive index function $n$ is supposed to satisfy $n\equiv 1$ in $|x|>R$ for some $R>0$. The incident wave
$v^{in}$ is allowed to be \mm{either a plane wave of the form}
\ben
v^{in}(x)=e^{ikx\cdot d},\quad d=(\cos\theta,\sin\theta)^T\tcr{\in \Sb}
\enn
\mm{where $d$ is a fixed incident direction $d$, and} \tcr{$\theta\in[0, 2\pi)$ is the incident angle},
or a point source wave emitting from the fixed source position $z$, taking the form
\ben
v^{in}(x)=\frac{i}{4}\,H_0^{(1)}(k|x-z|),\quad x\neq z,
\enn
\tcr{where $H_0^{(1)}(\cdot)$ is the Hankel function of the first kind of order zero.}
The scattered field \mm{$v^{sc}$} is required to fulfill the Sommerfeld radiation condition
\ben
\sqrt{r}\left(\partial_r v^{sc}-ik v^{sc}\right)\rightarrow 0\quad\mbox{as}\quad r=|x|\rightarrow \infty
\enn
\tcr{uniformly in all directions  $\widehat{x}:=x/r$},
leading to the far-field pattern $v^\infty$ in the asymptotic behavior
\ben
v^{sc}(x)=\frac{e^{ikr}}{\sqrt{r}}\left(v^\infty(\widehat{x})+O\lb\frac{1}{r}\rb\right)\quad\mbox{as}\quad r\rightarrow\infty.
\enn
 Our aim in this part is to consider the inverse problem of determining the refractive index $n$ and the shape of the supp$(1-n)$ from the knowledge of $v^{\infty}(\widehat{x})$  for $\widehat{x}\in \Gamma_{1}$. We prove that \tcr{an adimissible set of} $n$ and supp$(1-n)$  can be determined uniquely from the far field pattern \mm{$v^{\infty}$} of a single incident plane wave or \mm{a} point source.

\subsection{Literature review}
\mm{Inverse problems of partial differential equations (PDEs) is a very broad field of
various research directions, among which the inverse coefficients and/or obstacles problems
have recently attracted a tremendous attention, particularly from the mathematical point of view.}
\tcr{Inverse coefficient} problems for time-dependent PDEs \mm{were widely studied,
but still not much progress has been made for inverse transmission problems of recovering interfaces.}
\medskip

This paper will mainly be concerned with \mm{the simultaneous identification of both obstacles and coefficients}
appearing in some PDEs. More precisely, our \mm{main focus is on} the uniqueness issue for time-dependent problems with
a single dynamical  data. There is a wide mathematical literature on this topic, but it is mostly concerned with the knowledge of
\mm{sufficiently large measurement data; for example, the entire} Dirichlet-to-Neumann map. \mm{We shall consider the cases
that are very important in applications, namely, only one single dynamical data is available},
and focus on the determination of \tcr{both the geometrical shape of convex penetrable scatterers} and some coefficients appearing in the wave,  the dynamical Schr\"odinger and the Helmholtz equations.
\medskip

\mm{The wave model \eqref{wave-equation}-\eqref{Initial-wave} may be used in many applications, such as
the thermoacoustic (TAT) and photoacoustic (PAT) tomographies;
see, e.g.,} \cite{Kuchment_Kunyansky_TAT,Kuchment_Book,Stefanov_TAT,Stefanov_Yang_TAT_PAT,Otmar_Book}
and references therein.  There have been some results about the determination of the sound speed or the source term in the wave equation from a single measurement data. \mm{A uniqueness result was studied in \cite{Hristova}
for determining the constant sound speed and in  \cite{Stefanov_Uhlmann_Source or sound}
for the recovery of the source term or the sound speed provided that one of them is known.
The unique recovery of both the sound speed and the source term was considered
in \cite{Finch_Hickmann_Transmission-TAT} for the case when the sound speed is radial.
Recently, the result of \cite {Liu_Uhlmann_sound-source-determination} was improved
in \cite{Christina_Moradifm_Single-meausremet} to more general coefficients, indicating that
the sound speed can be recovered from a single measurement provided it is a harmonic function.
We consider in this work the uniqueness result for simultaneously recovering both the sound speed and its convex polygonal
support from a single dynamical data.}  For more related works, we refer to \cite{Belishev_Kurylev_Boundary-control,Helin_Lassas_Oksanen_Single,HK2018,Liu_Oksanen_Stability,Oksanen_Uhlmann_Photoacoustic,Pestov_Reconstruction _sound-speed,Pestov_absorption and sound speed,Rakesh_Salo,Rakesh_Yuan_Recovering_initial_Values,Stefanov_Uhlmann_both_sound-source}.
\medskip

\mm{In regard with} the determination of potentials appearing in the dynamical Schr\"odinger equation,
there are many studies in the literature, but they are mostly concerned with the determination of
the potentials from \mm{infinitely many} boundary measurements.
We refer to, e.g., \cite{B,MC,BD,ib1,ib2,MYE,CE,Eskin,Sun,Ltzou} and references therein.  To the best of our knowledge,
\mm{we are not aware of any existing studies of the determination of potentials from a single dynamical data,
and this is one of the main motivations of this work.}
In the frequency domain, the determination of potentials appearing in the Schr\"odinger equation was studied in \cite{[Stef]},
\mm{where a uniqueness result was established for} recovering the small potentials from the knowledge of the scattering amplitude.
This result \mm{was improved later} in \cite{Ramm} for more general electric potentials.
\medskip

\mm{All the aforementioned studies} are concerned only with  inverse coefficient problems. For determining the shape of a sound-soft obstacle, the recent development \tcr{in the time domain} can be found in \cite{[L]},  in which an inverse obstacle problem for an acoustic transient wave equation is considered. The authors in \cite{[L]} proved a uniqueness result for the determination of a sound-soft obstacle from the lateral Cauchy data given on a subboundary of an open bounded domain. In \cite{[I1]} the uniqueness and stability issues for recovering penetrable or impenetrable obstacles from various boundary data \tcr{were considered}.
\mm{We shall restrict our consideration in this work} to determine the shape of a convex polygonal penetrable scatterer
using a single dynamical data. \mm{by means of the data along all the time,} we prove in the theory that the Laplace transform can be applied to the measurement data and the time-dependent inverse problem is thus transformed to an equivalent problem in the Fourier-Laplace domain with many parameters (frequencies).  To apply the recently developed shape identification theory \cite{EH2018, HLM}  to inverse coefficient problems in a corner domain, \mm{we shall show} that there exists at least one parameter (frequency) for which the Laplace-transformed solution cannot vanish at the corner points (see  Lemma \ref{Lemma about uniqueness for source term},
also \cite{Ikehata_Reconstruction_source}).
We will then prove via asymptotic analysis that this parameter in two dimensions can be \mm{taken sufficiently small}
for wave equations and \mm{sufficiently large} for the Schr\"odinger equation (see Lemmas  \ref{Lem:asy} and \ref{Lem:high}).
We refer to \cite{BCJ,ElHu2015,EH2018,Hu_Salo_Vesalainen_Shape-identification,Hu_Li_Zou_Uniquness_Maxwell-penetrable, [I2], Friedman_Isakov_Uniquenss_conductivity_single-measurement,Ikehata_Scattering-enclosure,Ikehata_Reconstruction_single-measurement,Ikehata_Reconstruction_source,HLM,Li_Hu_Yang_Interface-weakly singular} for \tcr{related works on target identification with a single measurement data in the stationary case}.
\medskip


\mm{The rest of the paper}s is organized as follow.
\mm{In Sections  \ref{Sec:Statement main result} and \ref{subsec:well-posedness wave}, we state our main results and present some preliminary results that will be used to prove the main theorems of the work.}
In Sections \ref{Sec:Proof of main theorem wave}, \ref{Sec: Proof of main Theorem Schrodinger} and  \ref{Sec:Proof of main theorem time-harmonic} we prove \mm{our main results, Theorems} \ref{Main Theorem wave}, \ref{THM2} and \ref{Main theorem time-harmonic} respectively. \tcr{Some concluding remarks and open problems are summarized in Section \ref{cr}.}

\section{Statement of the Main results}\label{Sec:Statement main result}
\noindent In order to state our main results, \mm{we first} introduce some notations.
\mm{For some $R>0$, let} $D_{i}\subset B_{R}$  for $i=1,2$ be two convex polygons. For any  corner point $\mathcal{O}$ of $\PD D_{i}$ for $i=1,2$,   we denote by $$B_{\epsilon}(\mathcal{O}):=\{x\in\R^{2}:\ \lvert x-\mathcal{O}\rvert <\epsilon\},$$ for some  $\epsilon>0$. We  denote  by $E$ a subset of $B_{R}$ satisfying $E\subset B_{R}\backslash{(\overline{D_{1}\cup D_{2}})}$.

\begin{center}
	\includegraphics[width=5.7cm,height=3.7cm]{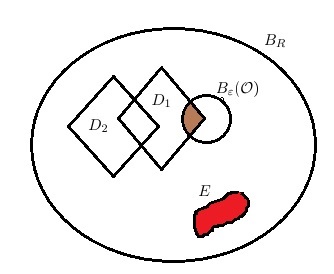}
\end{center}

\medskip
 For $A(x)=(a_{1}(x),a_{2}(x))\in \lb L^{\infty}(B_{R})\rb^{2}$ and $b\in L^{\infty}(B_{R})$, we  define \mm{a set $S(A,b)$ by}
 \[S(A,b):= \Big\{v(x): \Delta v(x)+A(x)\cdot\nabla v(x)+b(x)v(x)=0 \ \mbox{in} \ B_{R}\Big\}.\]

\subsection{\mm{Inverse wave problem}}\label{sec:wave} Here we state the main result for the inverse problem related to wave Equations \eqref{wave-equation} and \eqref{Initial-wave}. For $\alpha\in (0,1)$ we define the following admissible set of  coefficient $c$:
\begin{align*}
\mathcal{C}_{D}:=\left\{c\in C^{0,\alpha}\lb \overline{B_{\epsilon}(\mathcal{O})\cap D}\rb\ \mbox{for each corner $\mathcal{O}$ of $\PD D$ and some $\epsilon>0$}:\ supp(1-c)=D\right\}.
\end{align*}
\begin{theorem}\label{Main Theorem wave}
\mm{Let $(D_i, c_i)$ ($i=1,\,2$)} be two pairs of convex polygonal \tcr{scatterers} $D_{i}$ and  sound speeds $c_{i}$ such that $c_{i}\in\mathcal{C}_{D_{i}}$ for $i=1,2.$  \tcr{Let $f\in C_{0}^{\infty}(E)$ and $g\in C_{0}^{\infty}(0,T)$ be two non-vanishing functions}, \mm{and $u_{i}(x,t)$ be the \tcr{unique} solution to Equations \eqref{wave-equation} and \eqref{Initial-wave} with
$c$ replaced by $c_{i}$ for $i=1,2$.}
If
\begin{align}\label{Equality of given data}
u_{1}(x,t)=u_{2}(x,t),\ \mbox{for all} \quad x\in\Gamma_{R} \ \mbox{and} \ t>0\,,
\end{align}
then $D_{1}=D_{2}$. Moreover, Equation \eqref{Equality of given data} \mm{also} implies that $c_{1}=c_{2}$ provided the following conditions \tcr{hold}:
\begin{itemize}
	\item[(i)] $1/c^2_{i}(x)=V_{i}|_{\overline{\tcr{D_i}}}$ for $x\in \overline{\tcr{D_i}}$ and some function $V_{i}\in S(A,b)$ where $A$ and $b$ are analytic functions near the corner points of $\tcr{D_i}$.
	\item[(ii)] $\widehat{g}(0)\neq 0$ and $\int_{\R^2} f(y)\,dy\neq 0$, where \mm{$\widehat{g}$} denotes the Laplace transform of $g$.
\end{itemize}

\end{theorem}
\subsection{\mm{Inverse Schr\"odinger problem}}
In order to express the main statement of the second part of this work, \mm{we first} introduce  the  set of the admissible unknown compactly supported coefficients $q$.  For any  $\alpha\in(0,1)$ and $M\!>\!0$, we define an admissible set of $q$ by
\begin{eqnarray*}
\mathcal{Q}:=\mm{\Big\{} q: & \mbox{supp}(q)=D\subset B_R \ \mbox{is a convex polygon,}  ~q(\mathcal{O})\neq 0 \ \mbox{at each corner $\mathcal{O}$ of $\PD D$},  \\
& q\in \mathcal{C}^{0,\alpha}(\overline{B_{\varepsilon}(\mathcal{O})\cap D})
\ \mbox{for some $\epsilon>0$}
\mm{\Big\}}.
\end{eqnarray*}
%
\mm{Then we can state the uniqueness result for the inverse problem related to the Schr\"odinger equation below.}
\begin{theorem}\label{THM2}
\mm{Let $u_0\in \mathcal{C}_{0}^\infty(E)$,
$(D_i, q_i)$ ($i=1,\,2$)} be two pairs of convex polygonal obstacles $D_i$  and potentials $q_i\!\in\! \mathcal{Q}$,
\mm{and $u_i$ ($i=1,\,2$) be the solutions to the initial value problems (\ref{Eq1.1}) and (\ref{initial-Schrodinger})
with $q$ replaced by $q_i$.}
If
\begin{equation}\label{data}
u_1(x,t)=u_2(x,t)\qquad \mbox{on}\,\,\Gamma_R\times \R_+,
\end{equation}
then $D_1=D_2$. As a consequence,  \mm{under the following additional assertions:}
\begin{itemize}
\color{black} \item[i)] The coefficients $q_j\in S(A,b)$ for $j=1,\,2$ \tcr {where $A$ and $b$ are analytic functions near each corner point of $\tcr{D_i}$.}
\item [ii)]\tcr{There exists a $r_0>0$ such that $\displaystyle\int_{0}^{2\pi} u_0(r_0\cos\theta,r_0\sin\theta)\,d\theta\neq 0\,,$}\color{black}
\end{itemize}
the relation (\ref{data}) implies that $q_1=q_2$ in $D$.
\end{theorem}\label{thm:helmholtz}

\indent The above result claims the unique determination of the  \textit{obstacle} $D$ from the boundary measurement data $u_{|\Gamma_R\times\R_{+}}$ and  the recovery of the coefficient $q$ \mm{under some additional conditions.}

\subsection{\mm{Inverse Helmholtz problem}}
\mm{The main uniqueness result we will establish for the inverse Helmholtz problem can be stated in the following theorem.}

\begin{theorem}\label{Main theorem time-harmonic} \tcr{Let $k>0$ be fixed}, and
$u_{i}$ for $i=1,2$ be the solution to \eqref{Equation for v time-harmonic} \mm{when $n$ is replaced by $n_{i}$.} Suppose that
	\begin{itemize}
		\item[(i)]$n_{i}\in L^\infty(\R^2)$, $D_{i}=\mbox{supp}(n_{i}-1)$ is a convex polygon and $n_{i}\equiv 1$ in $\R^2\backslash\overline{D_{i}}$.
		\item[(ii)] $n_{i}(x)=V_{i}(x)|_{\overline{D_{i}}}$ for all $x\in \overline{D_{i}}$, where $V_{i}\in S(A, b)$ for some functions $A$ and $b$ which are analytic near each corner point of $D_{i}$, and $n_{i}(\mathcal{O})\neq 1$ for each corner point of $\partial D_{i}$.
		\item[(iii)] $\lvert u_{i}(\mathcal{O})\rvert>0$  for each corner point $\mathcal{O}$  of $\partial D_{i}$.
	\end{itemize}
	\mm{Then the equality} $u_{1}^{\infty}(\widehat{x})=u_{2}^{\infty}(\widehat{x})$ for all $\widehat{x}\in \tcr{\Sb}$ implies that $D_{1}=D_{2}$ and $n_{1}=n_{2}$.
\end{theorem}
\begin{remark}\tcr{In the time-harmonic regime,
the condition (iii) has already been used in \cite{Ikehata_Reconstruction_source} to prove uniqueness in recovering the support of the contrast function. For linear inverse source problems, the condition (ii) guarantees the unique identification of a source term having a convex-polygonal support (see \cite{HLM}).
\mm{The above Theorem\,\ref{thm:helmholtz}} verifies that both the support and the refractive index can be uniquely
identified under the conditions (i)-(iii).}
\end{remark}

\section{Preliminaries}\label{subsec:well-posedness wave}

\mm{An important idea in establishing our main results in this work is to transform the time-dependent problems
into the equivalent frequency dependent problems with the help of the Laplace transform.
The Laplace transform of a function $u$ of time variable $t$ is given  by}
\begin{equation}\label{Definition of Laplace transform}
\hat{u}(x,s):=\int_{0}^\infty e^{-st}\, u(x,t)\,dt,\quad  \mm{{s\in \C}, \,{\real{s}}>0,\,\,x\in\R^3.}
\end{equation}
Our goal in this section is to study the {\it{long time behavior}} of the solutions to the wave and Schr\"odinger equations,  in order to justify the use of the Laplace transform. \tcr{It is well known that such kind of behaviors can be derived from classical energy estimates, which will be presented below for the self-consistence of our arguments. Our emphasis will be placed upon the interpretation of the Laplace transform of $u$ in (\ref{Definition of Laplace transform}).}

\subsection{Long time behavior \mm{of solutions to} the wave equation}
\begin{lemma}\label{Lem:longtime}
	Let $F\in C^{\infty}_{c}\left(\R^2\times\R_{+} \right)$ \mm{satisfy
	that $\sup_{t\in[0,\infty)}\lVert\partial_{t}^{k}F(t,\cdot)\rVert_{L^{2}(\R^{+})}\leq C_{k}$
	for some constant  $C_{k}>0$ independent of $t$,} and let $u$ \tcr{be} a solution to the initial value problem
	\begin{align}\label{Equation of interest linear wave}
	\begin{aligned}
	\begin{cases}
	\frac{1}{c^{2}(x)}\partial_{t}^{2}u(x,t)-\Delta_{x} u(x,t)=F(x,t),\ \ (x,t)\in \R^2\times\R^{+}\\
	u(x,0)=\partial_{t}u(x,0)=0,\ \ x\in\mathbb{R}^{2}.\\
	\end{cases}
	\end{aligned}
	\end{align}
	\mm{Then the solution $u$ has the asymptotic estimate}
	\begin{align}\label{Estimate for Sobolev norm for u2}
	\Vert u(\cdot,t)\rVert_{L^{2}(\R^{2})}=O(t^{2})\quad\tcr{\mbox{as}\quad t\rightarrow \infty,}
	\end{align}
	\mm{leading to the well-definedness of the Laplace transform of $u$}
	for all $x\in \R^{2}$. \mm{Moreover, the following estimate holds:}
	\begin{align}\label{Estimate for Sobolev norm of Laplace transform for u}
	s^{3}\lVert \widehat{u}(\cdot,s)\rVert_{L^{2}(\R^{2})}\leq \C,
	\quad s>0,
	\end{align}
	for some constant $C>0$ independent of $s$.
	\begin{proof}
		Multiplying \tcr{the} first equation in \eqref{Equation of interest linear wave} by $2\partial_{t}u(x,t)$ and integrating over $(0,t)\times\R^{2}$, we get
		\begin{align*}
		\begin{aligned}
		&\int\limits_{\R^{2}}\left(\frac{1}{c^{2}(x)}\lvert\partial_{t}u(x,t)\rvert^{2}+\lvert\nabla_{x}u(x,t)\rvert^{2}\right)dx=2\int\limits_{0}^{t}\int\limits_{\R^{2}}F(x,s)\PD_{t}u(x,s)\D x\D s.\\
		\end{aligned}
		\end{align*}
\mm{By using the Cauchy-Schwartz inequality, the hypothesis \mm{on} $F(x,t)$, along with
the fact that  $0< c(x)\leq c$ for some constant $c>0$,  we obtain}
		\begin{align*}
		\begin{aligned}
		\tcr{\mathcal{E}}(t)&:=\int\limits_{\R^{2}}\lb\lvert\partial_{t}u(x,t)\rvert^{2}+\lvert\nabla_{x}u(x,t)\rvert^{2}\rb \D x\leq C\int\limits_{0}^{t}\lVert F(.,s)\rVert_{L^{2}(\R^{2})}\lVert\PD_{t}u(.,s)\rVert_{L^{2}(\R^{2})} \tcr{\D s}\\
		&\leq C\lVert F\rVert_{L^{\infty}\lb 0,\infty;L^{2}(\R^{2})\rb}\int\limits_{0}^{t}\tcr{\mathcal{E}}(s)^{1/2}\D s\leq C \int\limits_{0}^{t}\tcr{\mathcal{E}}(s)^{1/2}\D s
		\end{aligned}
		\end{align*}
		\mm{for some} constant $C>0$, depending only on $c(x)$ and $F(x,t)$.
		Now let us define
		\begin{align*}
		\phi(T):=\max_{0\leq t\leq T}\tcr{\mathcal{E}}(t).
		\end{align*}
		\mm{Using this, we derive from the above estimate that}
		\begin{align*}
		\phi(T)\leq C\max_{0\leq t\leq T}\int\limits_{0}^{t}\tcr{\mathcal{E}}(s)^{1/2}\D s\leq C\int\limits_{0}^{T}\tcr{\mathcal{E}}(s)^{1/2}\D s\leq C \int\limits_{0}^{T}\max_{0\leq s\leq T}\tcr{\mathcal{E}}(s)^{1/2}\D s=C \phi(T)^{1/2}T.
		\end{align*}
		This gives
		\[\tcr{\mathcal{E}}(T)\leq \phi(T)\leq CT^{2},\ \mbox{for any} \ 0\leq T<\infty.\]
	Also
		\begin{align*}
		\lVert u(\cdot,t)\rVert_{L^{2}(\R^{2})}=\Bigg\lVert\int\limits_{0}^{t}\partial_{t}u(\cdot,s)\D s\Bigg\rVert_{L^{2}(\R^{2})}\leq \int\limits_{0}^{t}\lVert\PD_{t}u(\cdot,s)\rVert_{L^{2}(\R^{2})}\D s\leq Ct^{2}.
		\end{align*}
		\mm{This proves \eqref{Estimate for Sobolev norm for u2}.} Combining the estimate of $\tcr{\mathcal{E}}(t)$ together with Equation \eqref{Estimate for Sobolev norm for u2}, we get that
		\begin{align}\label{H1 estimate on u}\lVert u(.,t)\rVert_{H^{1}(\R^{2})}\leq C t,\  \mbox{for some constant $C>0$ independent of $t$}.\end{align}

Next we will show that the Laplace transform defined in Equation \eqref{Definition of Laplace transform} makes sense for all $x\in\R^{2}$. To show this, \mm{it suffices} to prove that
		\begin{align*}
		\begin{aligned}
		\tcr{\lVert \widehat{u}(\cdot;s)\rVert_{L^{2}(\R^{2})}=}\Bigg\lVert\int\limits_{0}^{\infty}e^{-s t}u(\cdot,t)dt\Bigg\rVert_{L^{2}(\R^{2})}
		\end{aligned}
		\end{align*}
		is finite. \mm{By} using the Minkowskii inequality for inetgrals  (\cite{Folland_Measure_Theory}, page 194), we have
		\begin{align*}
		\begin{aligned}
		\Bigg\lVert\int\limits_{0}^{\infty}e^{-s t}u(x,t)dt\Bigg\rVert_{L^{2}(\R^{2})}&\leq \int\limits_{0}^{\infty}e^{-s t}\lVert u(\cdot,t)\rVert_{L^{2}(\R^{2})} dt.
		\end{aligned}
		\end{align*}
		Now using Equation \eqref{Estimate for Sobolev norm for u2}, we further derive
		\begin{equation*}
		\begin{aligned}
		\lVert \widehat{u}(\cdot;s)\rVert_{L^{2}(\R^{2})}&=\left\lVert \int\limits_{0}^{\infty}e^{-s t}u(x,t)\D t\right\rVert_{L^{2}(\R^{2})}\leq \int\limits_{0}^{\infty}e^{-s t}\left \lVert \widehat{u}(\cdot,t)\right\rVert_{L^{2}(\R^{2})}\leq C\int\limits_{0}^{\infty}e^{-s t}t^{2}\D t\leq \frac{C}{s^{3}}.
		\end{aligned}
		\end{equation*}
Thus
		\[\int\limits_{0}^{\infty}e^{-s t}u(x,t)\D t\in L^{2}(\R^{2}),\]
		which implies that
		\[\widehat{u}(x;s)=\int\limits_{0}^{\infty}e^{-s t}u(x,t)\D t \] exists  for almost every  $x\in\R^{2}$.		This completes the proof of \mm{Lemma\,\ref{Lem:longtime}.}
	\end{proof}
\end{lemma}
\begin{lemma}\label{Lem:neu}
Let $u(x,t)$ be the solution to \eqref{wave-equation}-\eqref{Initial-wave} when $c\in \mathcal{C}_{D}$, \mm{$f$ and $g$ are
two functions as given} in Theorem \ref{Main Theorem wave}.
Then the relation \mm{that} $u(x,t)=0$ for $(x,t)\in\Gamma_R\times \R_{+}$ implies $\partial_\nu u(x, t)=0$ for  $(x,t)\in \Gamma_R\times \R_{+}$.
\begin{proof}
	\mm{It is easy to see that Equations \eqref{wave-equation}-\eqref{Initial-wave} reduces to the
	following equations in $\lb\R^{2}\backslash B_{R}\rb\times (0,T)$:}
	\begin{equation}\label{Wave-equation outside BR}
	\begin{aligned}
	\begin{cases}
	\PD_{t}^{2}u(x,t)-\Delta_{x}u(x,t)=0, \ (x,t)\in\lb\R^{2}\backslash B_{R}\rb\times \R^{+}\\
	u(x,0)=\PD_{t}u(x,0)=0, \ \ \ \ x\in\R^{2}\backslash B_{R}.
	\end{cases}
	\end{aligned}
	\end{equation}
	Now after multiplying \eqref{Wave-equation outside BR} by $2\PD_{t}u(x,t)$ and integrating over $\lb\R^{2}\backslash B_{R}\rb\times\lb0,t\rb$, we get
	\begin{align*}
	\int\limits_{\R^{2}\backslash B_{R}}\lb \lvert\PD_{t}u(x,t)\rvert^{2}+\lvert\n_{x}u(x,t)\rvert^{2}\rb\D x=0, \ \mbox{for any $t\in\R^{+}.$}
	\end{align*}
	This gives
	\[u(x,t)=0, \ \mbox{for $\lb x,t\rb \in \lb\R^{2}\backslash B_{R}\rb\times\R^{+}$.}\]
Now using the fact that solution   $u\in C^{\infty}\lb \lb\R^{2}\backslash{B_{R}}\rb\times \R_{+}\rb$, we conclude that $\PD_{\nu}u(x,t)=0$ for any for $\lb x,t\rb\in \Gamma_{R}\times\R_{+}.$ This complete \mm{the proof of Lemma\,\ref{Lem:neu}.}
	\end{proof}
\end{lemma}
\subsection{Long time behavior \mm{of solutions to} the Schr\"odinger equation}
\mm{The following lemma indicates} that $\hat{u}$ is well defined and that  $\hat{u}(\cdot,s)\in H^2(\R^2)$.
\begin{lemma}\label{prop1.2}
\mm{Suppose that $u_0\in L^2(\R^3)$ and $q\in L^\infty(\R^3)$, and} there exists $M>0$ such that
\begin{equation}\label{Eq1.4}
\|u_0\|_{L^2(\R^3)}+\|q\|_{L^\infty(\R^3)}\leq M.
\end{equation}
\mm{Then the following estimate holds}
$$\|u(\cdot,t)\|_{L^{2}(\R^3)}\leq C (1+t),$$
\mm{where $C$ is a constant depending} only on $M$. \color{black}  Moreover, we have $\hat{u}(\cdot,s)\in H^2(\R^2)$.
\end{lemma}
\begin{proof}
From the Duhamel formula, \mm{we can express} the solution $u$ of \tcr{the initial value problem (\ref{Eq1.1})-(\ref{initial-Schrodinger})}
\mm{in the form} 
\begin{equation}\label{Eq1.9}u(x,t)=e^{i t\Delta}\,u_0(x)+i\int_{0}^t e^{i(t-s)\Delta} \,\,u(x,s) \,q(x)\; \tcr{d s.}
\end{equation}
Therefore, we readily get
$$\|u(\cdot,t)\|_{L^{2}(\R^3)}\leq \|u_0\|_{L^2(\R^3)}+\int_{0}^t   \|u(\cdot,s) \,q\|_{L^2(\R^3)}\,ds.$$
In light of (\ref{energy consevation}) and (\ref{Eq1.4}), we get
\begin{eqnarray}\label{Eq1.10}
\|u(\cdot,t)\|_{L^2(\R^3)}\leq \|u_0\|_{L^2(\R^3)}+M t \|u_0\|_{L^2(\R^3)} \leq C (1+t).
\end{eqnarray}
 In view of (\ref{Eq1.10}), one can see  that  $\hat{u}(\cdot,s )\in L^2(\R^3)$.
Indeed, we have
$$
\|\hat{u}(x,s)\|^2_{L^2(\R^3)}\leq \int_{0}^\infty e^{-2st}\,\|u(\cdot,t)\|^2_{L^2(\R^3)}\,dt\,\,<\infty.$$
Then,
$\hat{u}(\cdot,s)$ is well defined. Moreover, \mm{by} applying the Laplace transform, one has
$$\Delta \,\hat{u}(\cdot,s)=i\,u_0(\cdot)-is\,\hat{u}(\cdot,s)-q(\cdot)\,\hat{u}(\cdot,s)\in L^2(\R^2).$$
Therefore, $\hat{u}(\cdot,s)\in H^2(\R^2)$, which  completes the proof of \mm{Lemma\,\ref{prop1.2}.}
\end{proof}

Proceeding \mm{as we did in the proof of Lemma \ref{Lem:neu}, we can come to the following claim.}
\begin{lemma}
\label{Lemma 2}
\tcr{Let $s>0$ be fixed}, and $U\in H^{2}(\R^2\backslash\overline{B}_R)$ be a solution to the following elliptic equation
\begin{equation}\label{equation}
\Delta U(x)+is \,U(x)=0,\qquad \forall\,x\in\R^2\backslash \overline{B}_R.
\end{equation}
Then, we have
$$ \tcr{U(x)\Big|_{\Gamma_{R}}=0 \quad \mbox{implies \,\,that }\quad   \PD_{\nu}U(x)\Big|_{\Gamma_{R}}=0.}$$
\end{lemma}
\begin{proof} \tcr{Multiplying the equation (\ref{equation}) by $\overline{U}(x)$ and integrating by parts lead to $U(x)\equiv 0$ in $\R^2\backslash B_R$. This particularly implies the vanishing of the normal derivative of $U$ on $\Gamma_R$.}
\end{proof}

\section{Proof of Theorem \ref{Main Theorem wave}}\label{Sec:Proof of main theorem wave}
This section is devoted to the proof of \mm{our main results in Theorem \ref{Main Theorem wave},
separated in two subsections.
We first prove the unique identification of $\PD D$ in subsection \ref{sec:shape}, and
then show in subsection \ref{sec:coefficient} two lemmas which are used to establish the uniqueness for identifying $c(x)$.
The results of Theorem \ref{Main Theorem wave} are a consequence of the results from subsections \ref{sec:shape} and \ref{sec:coefficient}.
}
\subsection{Shape identification}\label{sec:shape}
Suppose that there are two convex polygonal obstacles $(D_1, c_1)$ and $(D_2, c_2)$  which generate the identical measurement data $u_1=u_2$ on $\Gamma_R\times \R_+$. We will show that $D_1=D_2$ in this subsection. Note that we have the following transmission \tcr{conditions} on $\partial D_j$:
\ben
u_j^+=u_j^-,\quad \partial_\nu u_j^+=\partial_\nu u_j^-  \quad\mbox{on}\quad \partial D_j \times \R^+,\quad j=1,2,
\enn
where the symbols $(\cdot)^\pm$ denote the limits taking from outside (+) and inside (-) of $D_j$ with respect to the space variable, respectively.

By Lemma \ref{Lem:longtime}, we can apply the Laplace transform to $u_j$ to obtain
\ben
\widehat{u}_j(x, s):=\int_{0}^\infty u_j(x,t)e^{-st}ds\in L^{2}(\R^{2}),\quad\mbox{for\, every positive number}\quad s>0.
\enn
Recalling the assumption that $u_1=u_2$ on $\Gamma_R \times\R_+$ and Lemma \ref{Lem:neu}, we obtain
\ben
\widehat{u}_1(x,s)=\widehat{u}_2(x;s),\quad   \partial_\nu\widehat{u}_1(x,s)=\partial_\nu\widehat{u}_2(x;s),\quad\mbox{on}\quad \Gamma_R. \enn
It is easy to deduce that $u_j$, \mm{for $j=1,2$, fulfils} the inhomogeneous elliptic equation
\ben
\Delta \widehat{u}_j(x,s)-p_j(x,s)\,\widehat{u}_j(x;s)=-f(x)\widehat{g}(s),\quad\mbox{in}\quad B_R
\enn for every fixed $s>0$, where $p_j(x,s):=s^2/ c_{j}^2(x)$.
\mm{Noting that $E\subset B_R\backslash (\overline{D_1\cup D_2})$ is the support of the spacial source term $f$,
we get $\widehat{u}_1=\widehat{u}_2$ in $B_R\backslash \overline{D_1\cup D_2\cup E}$
by the unique continuation of elliptic equations.}

If $D_1\neq D_2$, without loss of generality, we may assume that there exists a corner point $\mathcal{O}\in \partial D_1\backslash \partial D_2$ such that $B_\epsilon(\mathcal{O})\cap (\overline{D}_2\cup E)=\emptyset$ for some $\epsilon>0$.   Set $\Gamma:=B_\epsilon(\mathcal{O})\cap \partial D_1$. Then we obtain
\be\label{eq:1}\left\{\begin{array}{lll}
\Delta \widehat{u}_1(x,s)-p_{1}(x, s) \widehat{u}_1(x,s)=0,\quad&&\mbox{in}\quad B_\epsilon(\mathcal{O}),\\
\Delta \widehat{u}_2(x,s)-s^2 \widehat{u}_2(x,s)=0,\quad&&\mbox{in}\quad B_\epsilon(\mathcal{O}),\\
\widehat{u}_1(x,s)=\widehat{u}_2(x,s),\quad   \partial_\nu\widehat{u}_1(x,s)=\partial_\nu\widehat{u}_2(x,s)\quad&&\mbox{on}\quad \Gamma.
\end{array}\right.
\en Note that $\widehat{u}_2$ is analytic in $B_\epsilon(\mathcal{O})$ and $\widehat{u}_1\in H^2(B_\epsilon(\mathcal{O}))$.
\tcr{Since $c_1(\mathcal{O})\neq 1$, it holds that $p_1(\mathcal{O},s)\neq s^2$ for any $s>0$.}
Applying \cite[Lemma 1]{EH2018} we obtain $\widehat{u}_1=\widehat{u}_2\equiv 0$ in $B_\epsilon(\mathcal{O})$. By the unique continuation, we see $\widehat{u}_j(x,s)=0$ in $B_R\backslash\overline{D_j\cup E}$ for $j=1,2$ and every $s>0$.  Let $E^*\supset E$ be a neighbouring area of $E$ such that \tcr{$E^*\subset B_R$ and $E^*\cap D_1=\emptyset$}. Then the function $\widehat{u}_1$
satisfies
\begin{equation}\label{Helmholtz equation for u1 in E*}
\Delta \widehat{u}_1(x,s)-s^2\widehat{u}_1(x,s)=f(x)\widehat{g}(s)\quad\mbox{in}\quad E^{*}, \quad \widehat{u}_1\equiv 0 \quad \mbox{in}\quad E^*\backslash\overline{E}
\end{equation}
for all $s>0$.
Let $v$ be any solution \mm{to the equation} $\Delta v(x,s)-s^{2}v(x,s)=0$.  Now multiplying \eqref{Helmholtz equation for u1 in E*} by $v(x,s)$  and integrating over $\R^{2}$, we have for all $s>0$ that
\begin{eqnarray}
\widehat{g}(s)\int\limits_{E^{*}}f(x)v(x,s)\D x&=&\int\limits_{E^{*}}\lb \Delta \widehat{u}_{1}(x,s)-s^{2}\widehat{u}_{1}(x,s)\rb v(x,s)\D x
\nonumber \\
&=&  \tcr{ \int\limits_{E^{*}}\lb \Delta \widehat{u}_{1}(x,s)\,v(x,s)-s^{2}\widehat{u}_{1}(x,s)\,v(x,s)\rb \D x.}
\label{eq:E*}
\end{eqnarray}
\mm{Now using the integration by parts 
and noting} the vanishing of $\hat{u}_1$ near $\partial E^*$, we get \mm{from \eqref{eq:E*} that}
\[\widehat{g}(s)\int\limits_{\R^{2}}f(x){v}(x,s)\D x=0, \ \mbox{for all $s>0$ and $v(x,s)$ as \mm{specified} above.}\]
\tcr{Since $g$ does not vanish identically, there exists an open interval in which $\hat{g}\neq 0$.}
Now {for any $\omega\in {\Sb}$, taking the special solution $v(x,s)=e^{-sx\cdot\omega}$
to $\Delta v(x,s)-s^{2}v(x,s)=0$, we deduce
\[\widehat{f}(s\omega)=0, \ \mbox{for all $\omega\in\tcr{\Sb}$ and for $s>0$ belongs to some open interval.}\]
Since $f$ is compactly supported,  we have $f\equiv 0$ in $\R^{2}$, \mm{which is not true, hence completes
the proof that $D_1=D_2$.}

\subsection{Coefficient determination}\label{sec:coefficient}
Having proved that $D_1=D_2:=D$ in subsection \ref{sec:shape},
we can now verify that $c_1=c_2$ on $D$, under the additional conditions $(i)$ and $(ii)$
\mm{as stated in Theorem \ref{Main Theorem wave}.
To this purpose, we first present two auxiliary results in Lemmas\,\ref{Lemma about uniqueness for source term}
and \ref{Lem:asy}.}

\begin{lemma}\label{Lemma about uniqueness for source term}
\mm{For two given sets
\begin{eqnarray*}
\Sigma&:=&\{(r,\theta):\ \lvert \theta\rvert<\theta_{0},\ r<1 \ \mbox{and} \ \theta_{0}\in \lb0,\frac{\pi}{2}\rb\}\,, \\
\Gamma&:=&\{(r,\theta):\ \theta=\pm\theta_{0},\ r<1\}\,,
\end{eqnarray*}
suppose that $p\in C^{0,\alpha}(\overline{\Sigma})$, $h\in L^2(\Sigma)$,
$f\in C^{0,\alpha}(\overline{\Sigma})\cap H^{2}(\Sigma)$ with $f(\mathcal{O})\neq 0$,
and $u$ is a solution to the boundary value problem
\begin{align}\label{Equation of u in lemma}
\begin{aligned}
\begin{cases}
\Delta u(x)+p(x)u(x)=h(x)f(x),\ x\in{\Sigma}\,, \\
u(x)=\PD_{\nu}u(x)=0,\ x\in\Gamma\,.
\end{cases}
\end{aligned}
\end{align}
If the source component $h$ above belongs to $S(A,b)$ in a neighborhood of $\overline{\Sigma}$,
then $h$ is identically zero in $\Sigma$.
%
}
\end{lemma}
\begin{proof}
	Since $f\in H^{2}(\Sigma)\cap C^{0,\alpha}(\overline{\Sigma})$,  we have $f(x)=f(\mathcal{O})+\nabla_{x}f(\lambda \mathcal{O}+\lb 1-\lambda\rb x)$ for some $\lambda\in (0,1)$. \mm{Noting $f(\mathcal{O})\neq 0$ at $\mathcal{O}$,
	we know} that the lowest order expansion of $hf$ is harmonic. By using Lemma $2.3$ in \cite{HLM},
	we get $h(x)f(\mathcal{O})=0 $ for $x\in\overline{\Sigma}$,
	\mm{which implies $h\equiv 0$, hence completes the proof of Lemma\,\ref{Lemma about uniqueness for source term}.
	}
	\end{proof}
%
%
\begin{lemma}\label{Lem:asy} \mm{Let $u$ be the unique solution to
the initial value problem (\ref{wave-equation})-(\ref{Initial-wave}),
and $\widehat{u}$ be its Laplace transform of $u(x,t)$ with respect to the time variable.
Then for any corner point $\mathcal{O}$ of $\partial D$,}
there exists \tcr{a small number} $s_0>0$ such that $\widehat{u}(\mathcal{O},s_0)\neq 0$.
\end{lemma}
\begin{proof}
It is easy to check that
\begin{equation}\label{wave-equation after Laplace transform}
\Delta \widehat{u}(x,s)-s^2\widehat{u}(x,s)=-f(x)\widehat{g}(s)+s^2\lb \frac{1}{c^{2}(x)}-1\rb \widehat{u}(x,s)\quad \mbox{in}\quad \R^2.
\end{equation}
Recall that
\[\lb \Delta-s^{2}\rb \frac{i}{4}H^{(1)}_{0}\lb is\lvert x-y\rvert\rb=\delta(x-y),\ \mbox{for fixed $y\in\R^{2}$ and $x\neq y$}\]
where $H^{(1)}_{0}$ is the Hankel function  of first kind of order zero, \mm{which has the following asymptotic expansion
as $s\to 0$ (cf.\,\cite{[Nada]}):}
\begin{equation}\label{Asymptotic expansion for Hankel}
\frac{i}{4}H^{(1)}_{0}(is\lvert x-y\rvert)=-\frac{1}{2\pi}\ln\lvert x-y\rvert +\frac{i}{4}-\ln \frac{is}{2}-\frac{C}{2\pi}+O(s^{2}\lvert x-y\rvert^{2}\ln s\lvert x-y\rvert)\,.
\end{equation}
\mm{Then the solution to Equation \eqref{wave-equation after Laplace transform} can be} given by
\begin{align*}
\begin{aligned}
\widehat{u}(x,s)=-\frac{i}{4}\widehat{g}(s)\int\limits_{\R^{2}}f(y)H^{(1)}_{0}\lb is\lvert x-y\rvert\rb \D y +\frac{i}{4}s^{2}\int\limits_{\R^{2}}\lb \frac{1}{c^{2}(y)}-1\rb \widehat{u}(y,s)H^{(1)}_{0}(is\lvert x - y\rvert)\D y.
\end{aligned}
\end{align*}
Taking $x,y\in B_{R}$, \mm{sufficiently small $s>0$ and using the asymptotic expansion \eqref{Asymptotic expansion for Hankel},}
we get
	\begin{equation}\label{solution to wave-equation after Laplace equation}
	\begin{aligned}
\widehat{u}(x,s)&=-\frac{i}{4}\widehat{g}(s)\int\limits_{\R^{2}}f(y)\lb -\frac{1}{2\pi}\ln\lvert x-y\rvert +\frac{i}{4}-\ln \frac{is}{2}-\frac{C}{2\pi}\rb \D y\\
& \ \  +\frac{i}{4}s^{2}\int\limits_{\R^{2}}\lb \frac{1}{c^{2}(y)}-1\rb \widehat{u}(y,s)\lb-\frac{1}{2\pi}\ln\lvert x-y\rvert +\frac{i}{4}-\ln \frac{is}{2}-\frac{C}{2\pi}\rb\D y\\
& \ \ -\frac{i}{4}\widehat{g}(s)\int\limits_{\R^{2}}f(y)O(s^{2}\lvert x-y\rvert^{2}\ln s\lvert x-y\rvert)\D y \\
&\ \ +\frac{i}{4}s^{2}\int\limits_{\R^{2}}\lb \frac{1}{c^{2}(y)}-1\rb \widehat{u}(y,s)O(s^{2}\lvert x-y\rvert^{2}\ln s\lvert x-y\rvert)\D y.
	\end{aligned}
	\end{equation}
Multiplying Equation \eqref{solution to wave-equation after Laplace equation} by $s^{2}$ and using Equation \eqref{Estimate for Sobolev norm of Laplace transform for u}  together with the fact that $f,g$ and $\lb \frac{1}{c^{2}(x)}-1\rb$ are compactly supported, we
derive
\begin{equation}\label{limit after multiplying by $s^{2}$}
\lim_{s\rightarrow 0}s^{2}\widehat{u}(x,s)=0,\ \mbox{for  $x\in \overline{B_{R}}$}.
\end{equation}
\mm{Then noting that} $s^{2}\widehat{u}(x;s)$ is a continuous function for $x\in\R^{2}$ and $s>0$,
\mm{we know} there exist a constant $M>0$ such that
\begin{align}\label{Estimate after multiplying by $s^{2}$}
\lvert s^{2}\widehat{u}(x,s)\rvert\leq M, \ \mbox{for $s$ close to $0$ and $x\in\overline{B_{R}}$}.
\end{align}
Now multiplying Equation \eqref{solution to wave-equation after Laplace equation} by $s$ and using Equation \eqref{Estimate after multiplying by $s^{2}$}, we get
\begin{align}\label{limit after multiplying by $s$}
\lim_{s\rightarrow 0}s\widehat{u}(x,s)=0,\ \mbox{for $x\in B_{R}$}.
\end{align}
Using this and repeating the same argument \mm{as above, we know the existence of} a constant $M_{1}>0$ such that
\begin{align}\label{Estimate after multiplying by $s$}
\lvert s\widehat{u}(x,s)\rvert\leq M_{1}, \ \mbox{for $s$ close to $0$ and $x\in \overline{B_{R}}$}.
\end{align}
Finally using Equation \eqref{Estimate after multiplying by $s$} in \eqref{solution to wave-equation after Laplace equation} and the fact that $\widehat{g}(0)\neq 0$, $\int\limits_{\R^{2}}f(y)\D y\neq 0$,  \mm{we further deduce}
\begin{align*}
\lim_{s\rightarrow 0}\widehat{u}(x,s)=\infty, \ \mbox{for $x\in\overline{B_{R}}$}.
\end{align*}
\mm{This implies that for $\mathcal{O}\in\PD D$, there exists a sufficiently small $s_{0}$ such that
$\widehat{u}(\mathcal{O},s_{0})\neq 0,$
hence completes the proof of Lemma\,\ref{Lem:asy}.
}
\end{proof}
\begin{proof}[Proof of Theorem \ref{Main Theorem wave}]
\mm{The uniqueness for $\PD D$ was already established in subsection \ref{sec:shape}.
Next we prove the uniqueness of identifying $c(x)$, namely, $c_1(x)=c_2(x)$ for all $x\in \overline{D}$.
}
Taking the Laplace transform to $u_j$, it follows from the arguments as in subsection\,\ref{sec:shape}
that $\widehat{u}_1=\widehat{u}_2$ in $B_R\backslash\overline{D\cup E}$. Let $\mathcal{O}\in\partial D$ be a corner point and set $\Sigma:=B_\epsilon(\mathcal{O})\cap D$ for some small $\epsilon>0$. Then it is easy to see
\ben\left\{\begin{array}{lll}
\Delta \widehat{u}_1(x,s)-p_1(x,s) \widehat{u}_1(x,s)=0\quad&&\mbox{in}\quad \Sigma,\\
\Delta \widehat{u}_2(x,s)-p_2(x,s) \widehat{u}_2(x,s)=0\quad&&\mbox{in}\quad \Sigma,\\
\widehat{u}_1(x,s)=\widehat{u}_2(x,s),\  \partial_\nu\widehat{u}_1(x,s)=\partial_\nu\widehat{u}_2(x,s)\quad&&\mbox{on}\quad \Gamma:=B_\epsilon(\mathcal{O})\cap \partial D.
\end{array}\right.
\enn
Setting $w:=\widehat{u}_1-\widehat{u}_2$, we get
\begin{equation}\label{Equation for w differene of u1 and u2}
\begin{aligned}
\begin{cases}
&\Delta w-p_1(x,s)w=s^2 h(x) \widehat{u}_2(x,s)\quad\mbox{in}\quad \Sigma,\\
&w=\partial_\nu w=0\quad\mbox{on}\quad \Gamma,
\end{cases}
\end{aligned}
\end{equation}
where $h$ is defined by
\ben
h(x):=1/c_1^2(x)-1/c_2^2(x)=V_1(x)-V_2(x),\quad x\in \overline{\Sigma}.
\enn
\mm{By the assumption of $c_j$, we know $h\in S(A, b)$. Recalling the result in \cite{HLM},  we know
the lowest order expansion of $h$ near $\mathcal{O}$ is harmonic.
%
%
Furthermore, by Lemma \ref{Lem:asy} we know the existence of $s_{0}>0$ such that $\widehat{u}_2(\mathcal{O},s_{0})\neq 0$.
Then applying the Taylor series expansion for $\widehat{u}_2(x,s)$ around $\mathcal{O}$ leads to the fact that the lowest order expansion of $h(x)\widehat{u}_{2}(x,s_{0})$ in Equation \eqref{Equation for w differene of u1 and u2} belongs to $S(A,b)$.
Now using Lemma \ref{Lemma about uniqueness for source term}, we know $h\equiv 0$ in $\overline{\Sigma}$,
and the unique continuation argument further gives $c_{1}(x)=c_{2}(x)$ for $x\in\overline{D}$.
This completes the proof of Theorem \ref{Main Theorem wave}.
}
\end{proof}

\section{Proof of Theorem \ref{THM2}}\label{Sec: Proof of main Theorem Schrodinger}

\subsection{Shape identification}
Our goal in this subsection is to deal with the obstacle identification problem for the time dependent Schr\"odinger equation (\ref{Eq1.1}). More precisely, we aim  to prove that the measurement data  $u_{|\Gamma_R\times\R_+}$ can uniquely determine the object $D$ defined as the support of the coefficient $q$. \tcr{We will \mm{make some appropriate changes of} the proof of Theorem
\ref{Main Theorem wave} for the wave equation to be applicable to the Schr\"odinger equation.}

Let us consider two convex  polygonal obstacles   $D_1$ and $D_2$ corresponding to the two electric potentials $q_1$ and $q_2$
respectively. 
Let $u_1$ and $u_2$ be two respective solutions to the initial value problem(\ref{Eq1.1}) and (\ref{initial-Schrodinger})
for the Schr\"odinger equations corresponding to the coefficients $q_1$ (with support $D_1$) and $q_2$ (with support $D_2$).
After applying the Laplace transform, one can see that for any fixed $s>0$, the solutions  $\hat{u}_j$ for $j=1,\,2$ satisfy
$$\Delta \hat{u}_j(x,s)+(q_j(x)+is) \,\hat{u}_j(x,s)=i\,u_0(x),\quad \mbox{for all} \ \,x\in B_R$$
and $\hat{u}_1=\hat{u}_2$ on $\Gamma_R$ for any fixed $s>0$.
Let us recall that the function $u_0$ satisfies supp$(u_0)\subset E$. 
In view of the proof of Lemma \ref{Lemma 2} we obtain
 $\hat{u}_1=\hat{u}_2$ in $(\R^2\backslash B_R)\times \R_{+}$. 
\noindent Thus, by the unique continuation principle for elliptic equations, one can see that for any fixed $s>0$, we have
the following identity
\begin{equation}\label{3.8}
\hat{u}_1(x,s)=\hat{u}_2(x,s)\quad \mbox{in}\,\,B_R\backslash (\overline{D_1\cup D_2\cup E}).\color{black}
\end{equation}
\noindent On the other hand, by assuming that $D_1\!\neq \!D_2$, one can see that there exists (without loss of generality) a corner point $\mathcal{O}\in \partial D_1\backslash \partial D_2$.
For  $\varepsilon>0$, \mm{we recall that} $B_{\varepsilon}(\mathcal{O})$ is the ball centred in $\mathcal{O}$ satisfying $$B_{\varepsilon}(\mathcal{O})\cap (E\cup \overline{D_2})=\emptyset,$$
and  $\Gamma:=\partial D_1 \cap B_{\varepsilon}(\mathcal{O})$.
Since   $D_2\cap B_\varepsilon(\mathcal{O})=\emptyset$, then for any fixed $s>0$, $u_1$ and $u_2$ satisfy
\begin{equation}\label{+}\Delta \hat{u}_1(x,s)+(q_1(x)+is)\,\hat{u}_1(x,s)=0,\quad \mbox{and}\,\,\,\,\Delta \hat{u}_2(x,s)+is \,\hat{u}_2(x,s)=0,\quad \mbox{for all}\ \,x\in B_\varepsilon(\mathcal{O}).
\end{equation}
\noindent Moreover, taking into  account the fact that  $\hat{u}_j(\cdot,s)\in H^2(\R^2)$ for $j=1,\,2$, one can see in light of (\ref{3.8}) and  Lemma \ref{Lemma 2} that  we have for any fixed $s>0$,
 \begin{equation}\label{3.10}
\hat{u}_1(x,s)=\hat{u}_2(x,s),\quad \mbox{and}\,\,\,\,\,
\PD_\nu \hat{u}_1(x,s)=\PD_\nu \hat{u}_2(x,s),\quad \,\mbox{for all}\  \,x\in\Gamma.
\end{equation}
\tcr{By our assumption, $q_1(O)\neq 0$. }
Now, applying \cite[Lemma 1]{EH2018} (see also
Lemma \ref{Lemma about uniqueness for source term}) to the Cauchy problem
  (\ref{+}) and (\ref{3.10}),  we obtain the following identity \tcr{(cf. (\ref{eq:1}) in the wave equation case)}:
\begin{equation}\label{3.13}
\hat{u}_1(x,s)=\hat{u}_2(x,s)=0,\quad\,\, \forall x\in B_\varepsilon(\mathcal{O})
\end{equation}
for any fixed $s>0$. In view of the unique continuation principle, we have
\begin{equation}\label{3.14}
\hat{u}_1(x,s)=0,\quad \forall \,x\in B_R\backslash (\overline{D_1\cup E}).
 \end{equation}
 To derive the desired contradiction,  we still denoted by $E^*$ a neighborhood of  $E$ that satisfies the conditions
$$E\subset E^*\subset B_R,\quad \mbox{and}\,\,\,\,\, E^*\cap D_1=\emptyset.$$
Let $v$ be an arbitrary solution to the homogeneous equation
$$\Delta v(x,s)+is\,v(x,s)=0.$$
Multiplying $v$ to the equation of $\hat{u}_1$:
\begin{equation}\label{*}
\Delta \hat{u}_1(x,s)+is\,\hat{u}_1(x,s)=i\,u_0(x),\qquad \mbox{for all}\  \,x\in B_R
\end{equation}
and integrating over $E^*$, one gets the following identity
\begin{eqnarray*}
i\,\int_{E^*} u_0(x)\,v(x,s)\,dx&=&
\int_{E^*}(\Delta\hat{u}_1(x,s)+is \,\hat{u}_1(x,s))v(x,s)dx\\
&=&\int_{\PD E^*} (\PD_\nu \hat{u}_1(x,s)v(x,s)-\PD_\nu v(x,s)\hat{u}_1(x,s))\,d s.
\end{eqnarray*}
Now, \mm{taking} $v(x,s)=i\,e^{-i\sqrt{s}\,x\cdot \theta}$ with $\theta\in \mathbb{S}$ \mm{above} and using (\ref{3.14}), we get
$$ \mathcal{F}(u_0)(\sqrt{s}\,\theta)= \int_{\R^2} u_0(x) e^{-i\sqrt{s}\,\theta\,x}\,dx=0,\quad \mbox{for all}\  \,\theta\in\mathbb{S},\,\,\,s>0,$$
 where $\mathcal{F}(\cdot)$ denotes the Fourier transform of $u_0$. This implies that $u_0\equiv 0$ in $\R^2$, which is a contradiction. \mm{Thus we have proved} $D_1=D_2$.
\subsection{Coefficient identification}

Our  goal in this subsection is to pursue the proof of Theorem \ref{THM2}.
\mm{For $q_1,\,q_2\in\mathcal{Q}$, having already proved that $D_1=D_2=:D$, we now move forward}
to show that $q_1=q_2$ in $D$.
We start first with \mm{one of the main key ingredients in our proof.}
 \begin{lemma}\label{Lem:high}
Let $\mathcal{O}\in\PD D$ be a corner point, $\hat{u}$ be the solution to the equation
  \begin{equation}\label{4.7}\Delta \,\hat{u}(x,s)+(q(x)+is)\,\hat{u}(x,s)=i u_0(x),\quad \mbox{for all}\  \,x\in\R^2,\,s>0.
 \end{equation}
Then  there exists \mm{a sufficiently large} $s_0>0$ such that $\hat{u}(\mathcal{O},s_0) \neq 0$.
  \end{lemma}
\begin{proof}
Let us decompose the solution $\hat{u}$ into the sum $\hat{u}=\hat{v}+\hat{w}$, where $\hat{v}$ solves the equation
\begin{equation}\label{4.8}
\Delta \,\hat{v}(x,s)+is\,\hat{v}(x,s)=i\,u_0(x),\quad \mbox{for all}\ \,x\in\R^2,\,\,s>0,\end{equation}
and  $\hat{w}$ satisfies
\begin{equation}\label{4.9}
 \Delta \,\hat{w}(x,s)+is\,\hat{w}(x,s)+q(x)\,\hat{w}(x,s)=0,\quad  \mbox{for all}\ \, x\in\R^2,\,\,s>0.
 \end{equation}
 \tcr{Note that $v(x,t)$ satisfies the initial value problem (\ref{Eq1.1})-(\ref{initial-Schrodinger})  for the Schr\"odinger equation with $q\equiv 0$ in $\R^2$ and $w:=u-v$ denotes the scattered field incited by the potentials $v(x,t)$ and  $q(x)$. }
The proof  will be divided into two steps.\\
\noindent \textbf{Step 1:}  We prove that there exists a large $s_0\!>\!0$ such that $\hat{v}(\mathcal{O},s_0)\neq 0$.  Indeed, one can easily see that \mm{the solution $\hat{v}$}  to (\ref{4.8})  solves the following integral equation
$$\hat{v}(x,s)=\int_{E} i\,\phi(x,y;s)\,u_0(y)\,dy,$$
where the function $\phi$ is given by
\begin{equation}\label{phi}
\phi(x,y;s)=\frac{i}{4} H^{(1)}_0(\sqrt{is}\,|x-y|)=\frac{i}{4}H^{(1)}_0(\sqrt{s}e^{i\pi/4}|x-y|).
\end{equation}
Here $H^{(1)}_0$ denotes the Hankel function of  the first kind of order zero. This yields the following identity
$$\hat{v}(x,s)=\int_{E} \frac{-1}{4}\, H^{(1)}_0(\sqrt{s}\,e^{i\pi/4}|x-y|)\,u_0(y)\,dy.$$
Then, by taking  $s$  to the infinity, we will get in view of  \color{black} the asymptotic behavior of $H^{(1)}_0$ at infinity,
together with the identity (47) in \cite{[Nada]}, \color{black} that only the principal part of $\hat{v}$ will dominate and it will be equivalent to
\begin{eqnarray}\label{eq:2}
\hat{v}(x,s)\simeq -\frac{\sqrt{2}\,e^{-i\pi/4}}{4\sqrt{\pi\,e^{i\pi/4}}}\,\int_{\R^2}\frac{ u_0(y)}{\sqrt{\sqrt{s}\,|x-y|}}\,\,e^{i \sqrt{s}\,|x-y|\,\big(\frac{1}{\sqrt{2}}+i\frac{{1}}{\sqrt{2}}\big)}\,dy,\qquad \mbox{as}\,\,{s\rightarrow \infty}.
\end{eqnarray}
Thus, at the corner point $\mathcal{O}$ which is assumed to be, without loss of generality,  the origin of $\R^2$,  we have
\begin{eqnarray}\label{4.19}\hat{v}(\mathcal{O},s)\displaystyle & \simeq & -\frac{\sqrt{2}\,e^{-i\pi/4}}{4\sqrt{\pi\,e^{i\pi/4}}}\,\int_{\R^2}\frac{ u_0(y)}{\sqrt{\sqrt{s}|y|}}\,\,e^{i\sqrt{s}\,|y|\,\big(\frac{1}{\sqrt{2}}+i\frac{1}{\sqrt{2}}\big)}\,dy,\\ \nonumber
&= & - C_0 \,\int_{\R^2}\frac{u_0(y)}{\sqrt{\sqrt{s}|y|}}\,\,e^{i\sqrt{s}\,|y|\,\big(\frac{1}{\sqrt{2}}+i\frac{1}{\sqrt{2}}\big)}\,dy
\end{eqnarray}
as $s\rightarrow \infty$, where $C_0\in \C$ is a constant.
Let us denote by $$\tcr{I(s)}:=\int_{\R^2}\frac{ u_0(y)}{\sqrt{\sqrt{s}|y|}}\,\,e^{i\sqrt{s}\,|y|\,\big(\frac{1}{\sqrt{2}}+i\frac{1}{\sqrt{2}}\big)}\,dy. $$
In polar coordinates with $y=(r\cos \theta, r\sin\theta)$ and $r=|y|$, the integral $I$ will be given by
\begin{eqnarray*}
\tcr{I(s)}&=& s^{\frac{-1}{4}}\int_{0}^\infty e^{-(\frac{\sqrt{s}}{\sqrt{2}}-i\frac{\sqrt{s}}{\sqrt{2}})\,r}\,\,\sqrt{r}\,\Big(\int_{0}^{2\pi} u_0(r\cos\theta,r\sin\theta)\,\,d\theta\Big)\,\,\,dr\\
&=& s^{\frac{-1}{4}} \int_{0}^\infty e^{-Z_s\,r}\,\, g(r)\,\,dr\\
&:=& \tcr{s^{\frac{-1}{4}}}\,\hat{g}(Z_s),
\end{eqnarray*}
where $Z_s:= {\sqrt{s}}/{\sqrt{2}}-i {\sqrt{s}}/{\sqrt{2}}$ and $\hat{g}$ denotes the Laplace transform of  $g$
with respect to $r$, defined by
$$g(r):= \sqrt{r}\,\int_{0}^{2\pi} {u_0(r\cos \theta,r\sin\theta)}\,d\theta.$$
\tcr{Assume on the contrary that $\hat{v}(O,s)\equiv 0$ for \mm{all $s\geq M_0$ for a large constant
$M_0>0$.} Then the principal part $I(s)$ must also vanish for $s\geq M_0$, implying that
  $\hat{g}(Z_s)=0$ for $s\geq M_0$. Since $u_0$ has a compact support, we have $g(r)\equiv 0$ for large $r$. Therefore, the analyticity of $\hat{g}$
 leads to
 $g(r)=0$ for any $r\geq 0$. Consequently,
  $$\int_{0}^{2\pi} {u_0(r\cos \theta,r\sin\theta)} \,d\theta=0\quad \mbox{for\,\,any}\,\,r\geq 0,$$
 which contradicts the condition ii).} \\

 \noindent \textbf{Step 2:} \tcr{\mm{We prove that  the principal part of $\hat{u}(O,s)$ is
 dominated by $\hat{v}(O,s)$ as $s\rightarrow \infty$,} namely,  $\hat{w}(O,s)$ decays faster than $\hat{v}(O,s)$  when $s$ goes to the infinity. In the frequency domain, it is well known the solution $\hat{w}$ can be represented via the integral equation}
  $$\hat{w}(x,s)=\int_{B_R} \phi(x,y; s) \,q(y)\,\hat{u}(y,s)\,dy, $$
  where $\phi$ is given by $(\ref{phi})$.
  Define the integral operator
  $$\begin{array}{ccc}
  K_s: L^\infty(B_R)&\longrightarrow &L^\infty(B_R)\\
   \hat{u}&\longmapsto &\displaystyle\int_{B_R} \phi(x,y; s)\,q(y)\,\hat{u}(y,s)\,dy,\qquad \color{black}  \color{black}.
\end{array}$$
\tcr{Similar to (\ref{eq:2}),}
 we have for sufficiently large $s$ that
 \begin{equation}\label{what}\|\hat{w}\|_{L^{\infty}(B_R)}= \|K_s\,\hat{u}\|_{L^{\infty}(B_R)}\leq \frac{C}{s^{1/4}}\,\|\hat{u}\|_{L^{\infty}(B_R)}
 \end{equation}
  where $C$ depends on $M$ and  $B_R$. Thus, for large $s$ it holds that
 $$\|K_s\|< \frac{C}{s^{1/4}}<1.$$
 This entails that $(I-K_s)$ is an invertible operator. Since $\hat{u}=\hat{v}+\hat{w}=\hat{v}+K_s\hat{u}$,  then one can get
 \begin{equation}\label{uhat}\hat{u}=(I-K_s)^{-1} \hat{v}=\sum_{n=0}^\infty K_s^n \hat{u}=\hat{v}+K_s\hat{v}
 +\mm{\cdots}+K_s^{n}\hat{v}+\mm{\cdots}
 \end{equation}
We recall that $B_{\varepsilon}(\mathcal{O})$ is a neighborhood of the corner point $\mathcal{O}$.  Therefore,  for any $n\geq 1$ we have
\begin{equation}\label{estimation}\|K_s^n \hat{v}\|_{L^{\infty}(B_{\varepsilon}(\mathcal{O}))}\leq\|K_s^n \hat{v}\|_{L^{\infty}(B_R)}\leq \frac{C}{s^{n/4}} \|\hat{u}\|_{L^{\infty}(B_R)}\leq \frac{C}{s^{n/4}} \|\hat{v}\|_{L^{\infty}(B_R)}.
\end{equation}
In view of (\ref{estimation}), we can see that the second part $\hat{w}$ of the solution $\hat{u}$ decays faster to zero than $\hat{v}$ as $s$ intends to infinity.  This together with the first step completes the proof of \mm{Lemma \ref{Lem:high}}.
 \end{proof}

\begin{proof}[Proof of Theorem \ref{THM2}]
Let $\hat{u}_j$ ($j=1,2$) be the solution to the equation corresponding to   $q_j$, namely
$$\Delta \hat{u}_j(x,s)+(q_j(x)+is)\hat{u}_j(x,s)=i\,u_0(x),\quad  \forall\,x\in B_\varepsilon(\mathcal{O})\cap D.$$
Arguing like in the previous section, we can get that
 \begin{equation}\label{eqqq}\hat{u}_1=\hat{u}_2,\quad \mbox{and}\,\, \,\,\PD_\nu \hat{u}_1=\PD_\nu \hat{u}_2,\qquad \forall \,x\in \Gamma=B_\varepsilon\cap \PD D.
 \end{equation}
Denote $u:=u_1-u_2$. Then $u$ is a solution to
\begin{equation}\label{eqq}
\Delta u(x,s)-q_1(x) \,u(x,s)=(q_1-q_2)(x) u_2(x,s),\quad \forall\,x\in B_{\varepsilon}(\mathcal{O}).
\end{equation}
\tcr{By Lemma \ref{Lem:high}, we may chose a large $s>0$ such that $u_2(O,s)\neq 0$.  On the other hand,
since $q_1-q_2$ lies in the admissible set $S(A,b)$, by arguing analogously to the proof of Theorem \ref{Main Theorem wave}, we obtain $q_1=q_2$
 in $B_\varepsilon(\mathcal{O})$.
 \mm{Now applying the unique continuation argument concludes that} $q_1=q_2$ in $D$. }
\end{proof}

\begin{remark}
  \tcr{The analogue of Lemma \ref{Lem:high} for wave equations was proved via asymptotic analysis with a small number $s>0$ (see Lemma \ref{Lem:asy}). For the Schr\"{o}dinger equation, we shall carry out the proof by taking a large number $s>0$,
  \mm{as the proof of Lemma \ref{Lem:asy} does not apply to the Schr\"odinger equation.}}
  \end{remark}

	\section{Proof of Theorem \ref{Main theorem time-harmonic}}\label{Sec:Proof of main theorem time-harmonic}
	\begin{proof}[Proof of Theorem \ref{Main theorem time-harmonic}] \mm{We give a sketch of the proof}.
	Suppose  $D_1\neq D_2$. Without loss of generality, we may assume that there exists a corner point $\mathcal{O}\in \partial D_1\backslash \partial D_2$ and $B_\epsilon(\mathcal{O})\cap \overline{D}_2=\emptyset$ for some $\epsilon>0$.
	We recall \mm{the notations that} $\Sigma:=B_{\epsilon}(\mathcal{O})\cap D_{1}$ and  $\Gamma :=B_\epsilon(\mathcal{O})\cap \partial D_1$. Then we get
	\begin{equation*}
	\begin{aligned}
	\begin{cases}
	\Delta u_{1}+k^{2}n_{1}(x)u_{1}=0 \ \mbox{in}\ \Sigma,\\
	\Delta u_{2}+k^{2}u_{2}=0\ \mbox{in}\ \Sigma,\\
	u_{1}=u_{2},\ \PD_{\nu}u_{1}=\PD_{\nu}u_{2}\ \mbox{on} \ \Gamma,
	\end{cases}
	\end{aligned}
	\end{equation*}
	\mm{and the} difference $u:=u_{1}-u_{2}$ solves the Cauchy problem
	\begin{equation}\label{Equation for u time harmonic}
	\begin{aligned}
	\begin{cases}
	\Delta u+k^{2}n_{1}(x)u=k^{2}\lb 1-n_{1}(x)\rb u_{2}\ \mbox{in}\ B_{\epsilon}(\mathcal{O}),\\
	u=\PD_{\nu}u=0\ \mbox{on}\  \Gamma.
	\end{cases}
	\end{aligned}
	\end{equation}
	\mm{Noting that} $u_{2}$  solves $\Delta u_{2}+k^{2}u_{2}=0$ in $B_{\epsilon}(\mathcal{O})$,  its lowest order expansion around $O$ is harmonic. Then using Lemma \ref{Lemma about uniqueness for source term}, we get that $u_2\equiv 0$ in $\Sigma$,
	and further by unique continuation, $u_2\equiv 0$ in $\R^2$, which is impossible. Therefore we have $D_{1}=D_{2}:=D$.
	Now setting $\Sigma:=B_\varepsilon(\mathcal{O})\cap D$ and $\Gamma:= B_\varepsilon (\mathcal{O})\cap \PD D$, then
	\begin{equation}
	\begin{aligned}
	\begin{cases}
	\Delta u+k^{2}n_{1}(x)u=k^{2}\lb n_{2}-n_{1}\rb u_{2},\ \mbox{in} \ \Sigma\\
	u=\PD_{\nu}u=0,\ \mbox{on}\ \Gamma.
	\end{cases}
	\end{aligned}
	\end{equation}
	Since $n_{i}\in S(A,b)$ and $u_{2}(\mathcal{O})\neq 0$,  again applying Lemma \ref{Lemma about uniqueness for source term}  we get $n_{1}=n_{2}$ in $\R^2$.
\end{proof}

\section{Concluding remarks}\label{cr}
\mm{This work has been mainly devoted to the target identification and coefficient recovery problems for
the time-dependent wave and Schr\"odinger equations as well as the Helmholtz equation.
We have considered the penetrable scatterers with transmission conditions on the interface which are not much studied in the literature.  As we are interested in the important case when only a single dynamical data is available,
our investigations have been restricted to convex polygonal scatterers
and an admissible set of coefficients that include harmonic functions.

There are several interesting topics that deserve further investigation.
The first topic is the uniqueness for the important cases when the data is available only on a finite period of time
as well as for general penetrable scatterers. Note that our general idea of applying the Laplace transform relies heavily
on the data available over the infinite time, which can not be carried out to the case of the dynamical data
available only on a finite period of time.
However, we believe that the uniqueness results in recovering the shape (the first part in Theorems \ref{Main Theorem wave}  and \ref{THM2} ) can be generalized to non-polygonal convex penetrable scatterers,
while the Laplace transform provides the measurement data for each parameter (frequency).
The second topic is how to design efficient inversion algorithms for recovering convex polygonal scatterers
with only a single dynamical data,
based on the theory that has been developed here and some existing numerical schemes for time-harmonic inverse problems.
}

\mm{
\section*{Acknowledgements}
I. Ben A\"icha and M.  Vashisth were supported by the  NSAF grant (No. U1930402).
The work of J. Zou was substantially supported by Hong Kong RGC General Research Fund
(Project 14304517) and National Natural Science Foundation
of China/Hong Kong Research Grants Council Joint Research Scheme 2016/17 (project N\_CUHK437/16).
}

\end{document}